\documentclass[a4paper,12pt]{article}

\usepackage[T1]{fontenc}
\usepackage{amsmath, graphicx, color, amsthm, centernot, amssymb}
\usepackage[font=small,labelfont=bf,width=12.5cm]{caption}
\usepackage{algorithmic, algorithm,url}
\usepackage{lineno}
\usepackage[capitalise]{cleveref}

\usepackage[labelformat=simple]{subcaption}

\def\figscale{1.0}

\newenvironment{proofclaim}[1][]%
    {\noindent \emph{Proof.} {}{#1}{}}{$~$\hfill $~\blacklozenge$ \vspace{0.2cm}}

\newtheorem{theorem}{Theorem}[section]

\newtheorem{lemma}[theorem]{Lemma}
\newtheorem{corollary}[theorem]{Corollary}

\newtheorem{problem}[theorem]{Problem}

\newtheorem{claim}{Claim}%[section]

\newcommand{\set}[1]{\ensuremath{\left\{#1 \right\}}}

\newcommand{\inter}{\ensuremath{\mathrm{int}}}
\newcommand{\exter}{\ensuremath{\mathrm{ext}}}
\newcommand{\Pl}{\ensuremath{\mathrm{Pl}_{4,4f}}}

\begin{document}
%\linenumbers

\title{\textbf{Further Extensions of the Gr\"{o}tzsch Theorem}}

\author{	
	Hoang La\thanks{
		LIRMM, Universit\'{e} de Montpellier, CNRS, Montpellier, France
		E-Mail: \texttt{xuan-hoang.la@lirmm.fr}}, \
	Borut Lu\v{z}ar\thanks{Faculty of Information Studies in Novo mesto, Slovenia.\newline
		E-Mails: \texttt{borut.luzar@gmail.com, kennystorgel.research@gmail.com}}, \
	Kenny \v{S}torgel\footnotemark[2],\
}

\maketitle

{%
\abstract{
	The Gr\"{o}tzsch Theorem states that every triangle-free planar graph admits a proper $3$-coloring.
	Among many of its generalizations, the one of Gr\"{u}nbaum and Aksenov, giving $3$-colorability of
	planar graphs with at most three triangles, is perhaps the most known. 
	A lot of attention was also given to extending $3$-colorings of subgraphs to the whole graph.
	In this paper, 
	we consider $3$-colorings of planar graphs with at most one triangle.
	Particularly, we show that precoloring of any two non-adjacent vertices and precoloring of a face of length at most $4$
	can be extended to a $3$-coloring of the graph. 
	Additionally, we show that for every vertex of degree at most $3$, a precoloring of its neighborhood with the same color 
	extends to a $3$-coloring of the graph.
	The latter result implies an affirmative answer to a conjecture on adynamic coloring.
	All the presented results are tight.
}

}%

\section{Introduction}

A {\em proper coloring} of a graph $G$ is an assignment of colors to its vertices such that adjacent vertices are assigned distinct colors.
For an integer $k$, a graph is $k$-colorable if it admits a proper coloring with at most $k$ colors;
the smallest such $k$ is called the {\em chromatic number} of $G$, denoted by $\chi(G)$.

The Four Color Theorem~\cite{AppHak77,AppHakKoc77} states that the chromatic number of any planar graph is at most $4$,
but determining which graphs achieve the equality is an NP-complete problem~\cite{Dai80}.
Consequently, searching for properties of (planar) graphs that guarantee $3$-colorability is a very vibrant field 
(see, e.g.,~\cite{Bor13} for a survey).
It turns out that triangles play an important role in this decision problem;
indeed, a cornerstone theorem of Gr\"{o}tzsch~\cite{Gro59}
states that every triangle-free planar graph is $3$-colorable.
Consequently, the focus in the field turned to investigating ways in which triangles can appear in $3$-colorable planar graphs.
For example, for any plane triangulation, Heawood~\cite{Hea98} established a necessary and sufficient condition 
by showing that it is $3$-colorable if and only if all of its vertices have even degrees
(see~\cite{DikKowKur02,EllFleKocWen04,Koc18} for generalizations of this statement).

We may also allow triangles in general planar graphs and still retain $3$-colorability:
Havel~\cite{Hav69} conjectured that a $3$-colorable planar graph may contain arbitrarily many triangles as long as they are sufficiently far apart
and Steinberg~\cite{Ste93} conjectured that every planar graph without cycles of length $4$ and $5$ is $3$-colorable.
While Havel's conjecture has been proved by Dvo\v{r}\'{a}k, Kr\'{a}\v{l}, and Thomas~\cite{DvoKraTom16},
Steinberg's conjecture has been refuted by Cohen-Addad et al.~\cite{CohHebKraLiSal17}.
Currently the best result of a similar flavor is due to Borodin et al.~\cite{BorGleMonRas09}, 
stating that every planar graph without cycles of length $5$ and $7$, and without adjacent triangles is $3$-colorable. 
On the other hand, there are $3$-colorable planar graphs that may have close triangles (even incident) and have no short cycles forbidden:
as proved in~\cite{DroLuzMacSot19}, 
every planar graph obtained as a subgraph of the medial graph of a bipartite plane graph is $3$-colorable (in fact, $3$-choosable).

Another direction of research is focused on planar graphs with small number of triangles.
Gr\"{u}nbaum~\cite{Gru63} noticed that a planar graph may contain three triangles and still retain $3$-colorability.
His original proof was incorrect and a corrected version was published by Aksenov~\cite{Aks74}.
\begin{theorem}[\cite{Aks74}]
	\label{thm:3triangles}
	Every planar graph with at most three triangles is $3$-colorable.
\end{theorem}
Shorter proofs of this result were given by Borodin~\cite{Bor97} and Borodin et al.~\cite{BorKosLidYan14}.
The authors of the latter used the following result on $4$-critical graphs due to Kostochka and Yancey~\cite{KosYan14}.
\begin{theorem}[\cite{KosYan14}]
	\label{thm:4-critical}
	If $G$ is a $4$-critical graph on $n$ vertices, then 
	$$
		|E(G)| \ge \frac{5n-2}{3}\,.
	$$	
\end{theorem}
Theorem~\ref{thm:4-critical} is a restricted version of a more general theorem from~\cite{KosYan18},
which describes $k$-critical graphs %achieving the equality 
and was used in~\cite{BorDvoKosLidYan14} to characterize all planar $4$-critical graphs with exactly four triangles;
we present these two results in Section~\ref{sec:prel}.

Along with a short proof of Theorem~\ref{thm:3triangles}, using Theorem~\ref{thm:4-critical},
the authors of~\cite{BorKosLidYan14} presented short proofs of several other extensions of the Gr\"{o}tzsch Theorem, 
which guarantee $3$-colorability of graphs being close to triangle-free planar graphs.
In particular, they extended a result from~\cite{JenTho00} 
stating that a triangle-free planar graph with an additional vertex of degree $3$ is also $3$-colorable.
\begin{theorem}[\cite{BorKosLidYan14,JenTho00}]
	\label{thm:4vert}
	Let $G$ be a triangle-free planar graph and let $H$ be a graph such that $G = H - v$ for some vertex $v$ of degree $4$ of $H$. 
	Then $H$ is $3$-colorable.
\end{theorem}

They also gave a short proof of a precoloring extension result of Aksenov, Borodin, and Glebov~\cite{AksBorGle03}.
\begin{theorem}[\cite{AksBorGle03,BorKosLidYan14}]
	\label{thm:precolor2verts}
	Let $G$ be a triangle-free planar graph.
	Then each coloring of any two non-adjacent vertices can be extended to a $3$-coloring of $G$.
\end{theorem}
Note that Theorem~\ref{thm:precolor2verts} extends the result of Aksenov~\cite{Aks77} and Jensen and Thomassen~\cite{JenTho00}
that a graph obtained from a triangle-free planar graph by adding one edge is $3$-colorable.

From Theorems~\ref{thm:4vert} and~\ref{thm:precolor2verts} one can derive another precoloring extension result.
\begin{theorem}[\cite{BorKosLidYan14}]
	\label{thm:5face}
	Let $G$ be a triangle-free planar graph and let $f$ be a face of $G$ of length at most $5$.
	Then each $3$-coloring of $f$ can be extended to a $3$-coloring of $G$.
\end{theorem}

On the other hand, if the face $f$ has length $k$ with $k \ge 6$, 
then not every precoloring of its vertices can be extended to a $3$-coloring of $G$. 
The cases when $k = 6,7,8,9$ were completely characterized in~\cite{GimTho97}, \cite{AksBorGle04},
\cite{DvoLid15}, and~\cite{ChoEksHolLid18}, respectively.
Moreover, precoloring faces in planar graph of girth at least $5$ have also been studied (see, e.g., \cite{ChoEksHolLid18} for more details).

\bigskip
In this paper, we introduce new results about $3$-colorability of planar graphs with small number of triangles and some precolored vertices.
First, we extend Theorem~\ref{thm:precolor2verts} to planar graphs with at most one triangle.
\begin{theorem}
	\label{thm:triangle_plus_precolored_pair}
	Let $G$ be a planar graph with at most one triangle.
	Then each coloring of any two non-adjacent vertices can be extended to a $3$-coloring of $G$.
\end{theorem}
The result is tight in terms of the number of precolored vertices and in terms of the number of triangles;
for example, the precolorings of graphs depicted in Figure~\ref{fig:precolored_pair} cannot be extended to a $3$-coloring of the whole graph.
\begin{figure}[htb!]
	\centering
	\begin{subfigure}[b]{.5\textwidth}
		$$
			\includegraphics[scale=\figscale]{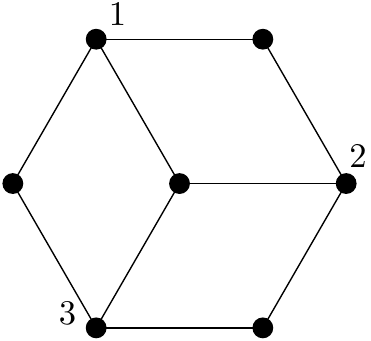}
		$$
		\caption{}
		%\label{fig:hexagon}
	\end{subfigure}	
	\begin{subfigure}[b]{.49\textwidth}
		$$
			\includegraphics[scale=\figscale]{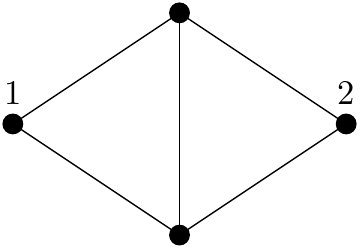}
		$$
		\caption{}
		%\label{fig:diamond}
	\end{subfigure}	
	\caption{Not every precoloring of three vertices can be extended to a $3$-coloring of a planar graph with at most one triangle (example (a)),
		nor can be every precoloring of two vertices in a planar graph with two triangles (example (b)).}
	\label{fig:precolored_pair}
\end{figure}
	
As a corollary of Theorem~\ref{thm:triangle_plus_precolored_pair}, we obtain a theorem similar to Theorem~\ref{thm:4vert} for planar graphs with at most one triangle.
\begin{theorem}
	\label{thm:triangle_plus_3vertex}
	Let $G$ be a planar graph with at most one triangle and let $H$ be a graph such that $G = H - v$ for some vertex $v$ of degree at most $3$ in $H$,
	which is adjacent with at most two vertices of the triangle in $G$ if it exists. 
	Then $H$ is 3-colorable.
\end{theorem}
\begin{proof}
	Let $N(v) = \set{v_1, v_2, v_3}$.
	Without loss of generality, we may assume that $v_1$ and $v_2$ are not adjacent.
	Then, by Theorem~\ref{thm:triangle_plus_precolored_pair}, we can color $v_1$ and $v_2$ with the same color
	and so the three vertices in $N(v)$ will be colored with at most two colors, 
	which means there is an available color for coloring $v$.
\end{proof}

Again, the result is tight in terms of the number of precolored vertices and in terms of the number of triangles 
(see Figure~\ref{fig:precolored_3vert} for examples),
as well as in terms of the number of neighbors of $v$ on the triangle.
Clearly, connecting $v$ with all three vertices of the triangle would result in a subgraph isomorphic to $K_4$.
\begin{figure}[htb]
	\centering
	\begin{subfigure}[b]{.5\textwidth}
		$$
			\includegraphics[scale=\figscale,page=2]{fig_3vert}
		$$
		\caption{}
		%\label{fig:4vert}
	\end{subfigure}	
	\begin{subfigure}[b]{.49\textwidth}
		$$
			\includegraphics[scale=\figscale,page=1]{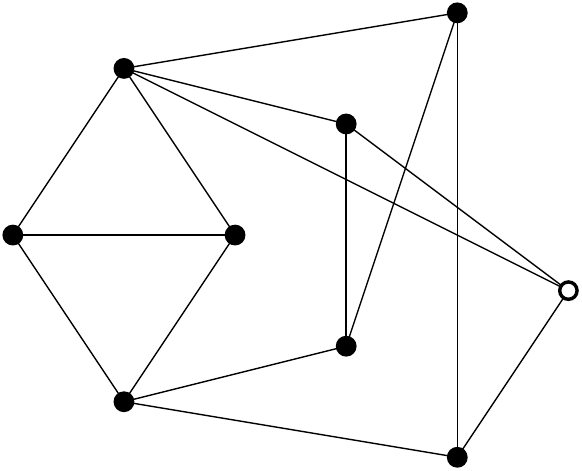}
		$$
		\caption{}
		%\label{fig:2trian}
	\end{subfigure}	
	\caption{Not every graph obtained from a planar graph with at most one triangle by adding a $4$-vertex is $3$-colorable 
		(example (a)),
		nor is a graph obtained from a planar graph with two triangles by adding a $3$-vertex 
		(example (b)).
		The added vertex is depicted with an empty disk in both cases.}
	\label{fig:precolored_3vert}
\end{figure}

Extending precolorings of small faces in planar graphs with one triangle is more restricted.
We prove an analogue of Theorem~\ref{thm:5face} for faces of length at most $4$.
\begin{theorem}
	\label{thm:onetriangle_5face}
	Let $G$ be a planar graph with at most one triangle and let $f$ be a face of $G$ of length at most $4$.
	Then each $3$-coloring of $f$ can be extended to a $3$-coloring of $G$.
\end{theorem}

On the other hand, a precoloring of a $5$-face in a planar graph with one triangle cannot always be extended to a $3$-coloring of the whole graph;
see example in Figure~\ref{fig:precolor5face}.
\begin{figure}[htb]
	$$
		\includegraphics{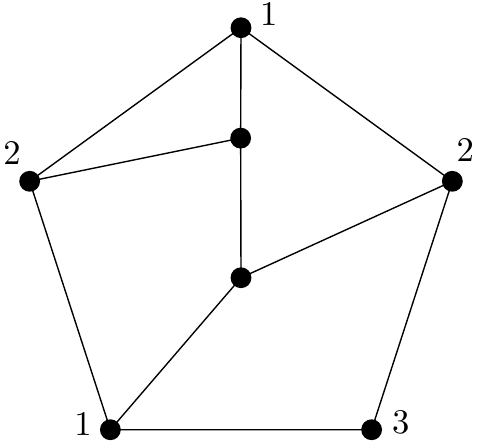}
	$$
	\caption{Precoloring of the outer $5$-face which cannot be extended to a $3$-coloring of the graph.}
	\label{fig:precolor5face}
\end{figure}

\medskip
The result about extending a precoloring of an $8$-cycle from~\cite{DvoLid15} (as remarked in~\cite{ChoEksHolLid18})
implies the following.
\begin{theorem}[\cite{DvoLid15}]
	\label{thm:4common}
	Let $G$ be a triangle-free planar graph and let $v$ be a vertex of degree at most $4$ in $G$. 
	Then there exists a $3$-coloring of $G$ where all neighbors of $v$ are colored with the same color.
\end{theorem}
A similar result to Theorem~\ref{thm:4common} about coloring three neighbors of a vertex of an arbitrary degree 
can be obtained as a corollary of Theorem~\ref{thm:4vert}.
\begin{corollary}
	Let $G$ be a triangle-free planar graph and let $v_1$, $v_2$, and $v_3$ be distinct vertices with a common neighbor $v$.
	Then there exists a $3$-coloring of $G$ where $v_1$, $v_2$, and $v_3$ are colored with the same color.
\end{corollary}
\begin{proof}
	By Theorem~\ref{thm:4vert}, the graph obtained from $G$ by adding a $4$-vertex $x$ adjacent to $v$, $v_1$, $v_2$, and $v_3$ is $3$-colorable.
	In its coloring, the vertices $v_1$, $v_2$, and $v_3$ are colored with the same color, since they must all be colored differently from $v$ and $x$,
	which receive two distinct colors.
\end{proof}

We prove a somewhat weaker result for the case of planar graphs with one triangle.
Let $K_4'$ be the graph obtained from $K_4$ by subdividing once the three edges incident with a vertex $v$ (see Figure~\ref{fig:K4sub}).
We call a graph {\em $K_4'$-free} if it does not contain $K_4'$ as a subgraph in such a way 
that the vertex $v$ of $K_4'$ has degree $3$ also in $G$.
It is easy to see that the vertices in the neighborhood of $v$ cannot be colored with a same color.
\begin{figure}[htb]
	$$
		\includegraphics{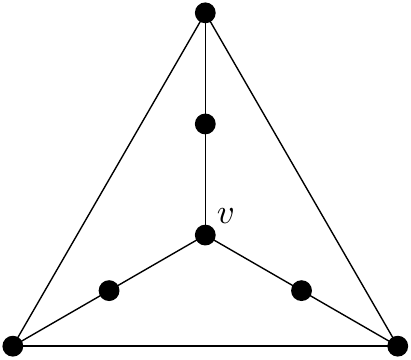}
	$$
	\caption{A planar graph with at most one triangle 
		with a vertex $v$ of degree $3$ having an independent neighborhood $N(v)$ for which there is no $3$-coloring such 
		that all vertices in $N(v)$ receive the same color.}
	\label{fig:K4sub}
\end{figure}

\begin{theorem}
	\label{thm:3neighbors}
	Let $G$ be a $K_4'$-free planar graph with at most one triangle.
	Then, for every vertex of degree at most $3$ with an independent neighborhood,
	a precoloring of its neighbors with the same color can be extended to a $3$-coloring of $G$. 
\end{theorem}

Theorem~\ref{thm:3neighbors} is tight in terms of the degree of a vertex and in terms of the number of triangles 
(see examples in Figure~\ref{fig:thm11_tight}).
\begin{figure}[htb]
	\centering
	\begin{subfigure}[b]{.5\textwidth}
		$$
			\includegraphics[scale=\figscale]{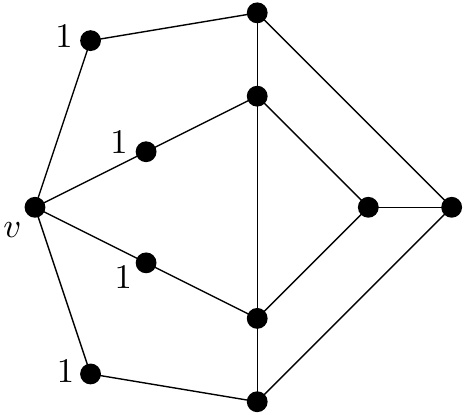}
		$$
		\caption{}
		\label{fig:4vert}
	\end{subfigure}	
	\begin{subfigure}[b]{.49\textwidth}
		$$
			\includegraphics[scale=\figscale]{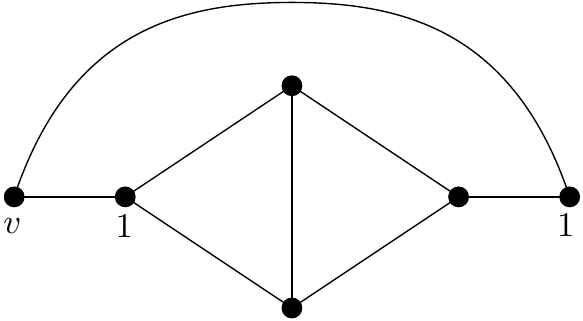}
		$$
		\caption{}
		\label{fig:2trian}
	\end{subfigure}
	\caption{
		Precoloring of the neighborhood of a $4$-vertex $v$ in a $K_4'$-free planar graph $G$ with one triangle
		cannot always be extended to a $3$-coloring of $G$ (example (a)).
		Similarly, precoloring of the neighborhood of a $2$-vertex $v$ in a planar graph $G$ with two triangles
		cannot always be extended to a $3$-coloring of $G$ (example (b)).}
	\label{fig:thm11_tight}
\end{figure}

\section{Preliminaries}
	\label{sec:prel}

In this section we present the terminology and the auxiliary results that we are using in the proofs of our theorems.

Note that we only consider simple graphs, i.e., loopless graphs without parallel edges;
thus, whenever we perform identification of vertices in our proofs, we discard eventual parallel edges.

For a graph $G$, we denote the number of its vertices and edges by $n_G$ and $m_G$, respectively.
If $G$ is a {\em plane graph}, i.e., a planar graph embedded in the plane, 
we denote the set of its faces by $F(G)$ and their number by $f_G$;
in particular, the number of faces of length $k$ is denoted by $f_{k,G}$
or simply $f_k$ if $G$ is evident from the context.
The length of a face $\alpha$ is denoted by $\ell(\alpha)$.
A vertex of degree $k$ (resp., at least $k$) is called a {\em $k$-vertex} (resp., a {\em $k^+$-vertex}), 
and similarly, a face of length $k$ is called a {\em $k$-face}.

We denote the graph obtained from a graph $G$ by deleting a vertex $v$ (resp., an edge $e$) by $G - v$ (resp., $G - e$).
A graph is {\em $k$-critical}, if $\chi(G) = k$ and for any $x \in V(G) \cup E(G)$, $\chi(G - x) < k$.
A subgraph of $G$ induced by a set of vertices $U$ is denoted by $G[U]$.

For a given cycle $C$ in a plane embedding of a graph $G$, we define $\inter(C)$ 
to be the graph induced by the vertices lying strictly in the interior of $C$.
Similarly, $\exter(C)$ is the graph induced by the vertices lying strictly in the exterior of $C$. 
A {\em separating} cycle is a cycle $C$ such that $\inter(C) \neq \emptyset$ and $\exter(C) \neq \emptyset$.

\medskip
The following lemma is a crucial tool in the proofs, where we use minimality of counterexamples;
see, e.g.,~\cite{BorKosLidYan14} for its proof.
\begin{lemma}[Lemma 10, Borodin]
	\label{lem:10Borodin}
	Let $G$ be a plane graph and $F = v_1v_2v_3v_4$ be a $4$-face in $G$ such that \hbox{$v_1v_3$, $v_2v_4\notin E(G)$}. 
	Let $G_i$ be obtained from $G$ by identifying $v_i$ and $v_{i+2}$ where $i \in \{1,2\}$. 
	If the number of triangles increases in both $G_1$ and $G_2$, 
	then there exists a triangle $v_iv_{i+1}z$ for some $z \in V(G)$ and $i\in \{1,2,3,4\}$. 
	Moreover, $G$ contains vertices $x$ and $y$ not in $F$ such that $v_{i+1}zxv_{i+3}$ and $v_izyv_{i+2}$ are paths in $G$ 
	(indices are modulo $4$).
\end{lemma}
\begin{figure}[htb!]
	$$
		\includegraphics{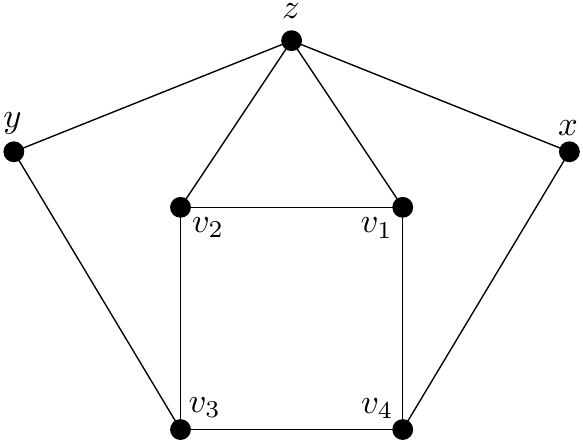}
	$$
	\caption{The configuration in Lemma~\ref{lem:10Borodin} in the case of a graph with one triangle.}
	\label{fig:10Borodin}
\end{figure}

In the case of planar graphs with one triangle, 
we can use the following simpler statement of Lemma~\ref{lem:10Borodin}.
\begin{corollary}
	\label{cor:lem10}
	Let $G$ be a plane graph with at most one triangle and let $\alpha$ be any $4$-face of $G$. 
	Then, at least one of the following holds:
	\begin{itemize}
		\item[$(a)$] $\alpha$ is adjacent to a triangle, or
		\item[$(b)$] for at least one pair of opposite vertices of $\alpha$, we can identify them without creating any new triangles.
	\end{itemize}
\end{corollary}

Theorem~\ref{thm:3triangles} settles $3$-colorability of planar graphs with at most three triangles.
The smallest example of a planar graph with four triangles that is not $3$-colorable is the complete graph $K_4$.
Plane $4$-critical graphs with exactly four triangles have been completely characterized by Borodin et al.~\cite{BorDvoKosLidYan14}.
In their proofs, they used the following result of Kostochka and Yancey~\cite{KosYan18},
which is a stronger version of Theorem~\ref{thm:4-critical}.

\begin{theorem}[\cite{KosYan18}]
	\label{thm:4-critical_ore}
	If $G$ is a $4$-critical graph, then 
	$$
		m_G \ge \frac{5n_G - 2}{3}\,.
	$$ 
	Moreover, the equality is achieved if and only if $G$ is a $4$-Ore graph.
\end{theorem}

Here, a graph is {\em $k$-Ore} if it is obtained from a set of copies of $K_k$ by a sequence of DHGO-compositions,
where a {\em DHGO-composition} $O(G_1,G_2)$ of graphs $G_1$ and $G_2$ is the graph obtained through the sequence of the following steps:
delete some edge $xy$ from $G_1$, split a vertex $z$ of $G_2$ into non-isolated vertices $z_1$ and $z_2$,
and identify $x$ with $z_1$ and $y$ with $z_2$.

By a {\em $\Pl$-graph} we denote a planar graph with exactly four triangles and no $4$-faces.
A correlation between $4$-Ore graphs and $\Pl$-graphs was given in \cite{BorDvoKosLidYan14}.
\begin{theorem}[\cite{BorDvoKosLidYan14}]
	\label{thm:pl44f-graphs}
	Every $4$-Ore graph has at least four triangles. 
	Moreover, a $4$-Ore graph has exactly four triangles if and only if it is a $\Pl$-graph. 
\end{theorem}

\section{Proofs of Theorems~\ref{thm:triangle_plus_precolored_pair}, \ref{thm:onetriangle_5face}, and~\ref{thm:3neighbors}}
	\label{sec:precoloring_two_verts}

We prove Theorem~\ref{thm:triangle_plus_precolored_pair} in two steps. 
First, we consider the case when the two precolored vertices receive distinct colors,
which is equivalent to the statement of Theorem~\ref{thm:triangle_plus_edge}.
\begin{theorem}
	\label{thm:triangle_plus_edge}
	Let $G$ be a planar graph with at most one triangle and let $H$ be a graph such that $G = H - e$ for some edge $e$ of $H$. 
	Then $H$ is $3$-colorable.
\end{theorem}

\begin{proof}
	We prove the theorem by contradiction.
	Suppose that $H$ is a counterexample minimizing the number of vertices plus the number of edges 
	and let $G$ be a plane graph with at most one triangle such that $G = H - e$ for some edge $e$ of $H$.
	Note that since $G$ is planar and contains at most one triangle, it is $3$-colorable by Theorem~\ref{thm:3triangles}. 
	%Thus, we may assume that $e$ is not a parallel edge, otherwise $H$ is also $3$-colorable. 
	By Theorem~\ref{thm:precolor2verts}, we may assume that $G$ contains exactly one triangle $T$.
	Moreover, by the minimality, $H$ is $4$-critical.
	
	We consider five cases regarding $4$-faces in $G$.
	
	\medskip \noindent
	{\bf Case 1:} \textit{$G$ has at most two $4$-faces.} \quad
	By the Handshaking Lemma, we have
	$$
		2m_G = \sum_{\alpha \in F(G)} \ell(\alpha) \ge 3 + 4 \cdot f_{4,G} + 5\cdot(f_G - (1 + f_{4,G})) = 5f_G - 2 - f_{4,G}
	$$	
	(in the calculation, we assume that $T$ is a face, otherwise the lower bound on the number of edges would be even higher).
	Then, $5f_G \le 2m_G + 4$ and by applying the Euler's Formula and observing that $n_H = n_G$ and $m_H = m_G + 1$, 
	we infer that 
	$$
		10 = 5n_G - 5m_G + 5f_G \le 5n_G - 3m_G + 4 = 5n_H - 3(m_H - 1) + 4\,.
	$$
	Thus, 
	$$
		m_H \le \frac{5n_H - 3}{3}\,,
	$$ 
	a contradiction to Theorem~\ref{thm:4-critical}.

	\medskip \noindent
	{\bf Case 2:} \textit{$G$ has a $4$-face $\alpha = v_1v_2v_3v_4$ such that at most one vertex of $\alpha$ is incident with $T$ 
	and at most one vertex of $e$ is incident with $\alpha$.} \quad
	Let $G_i$ be the graph obtained from $G$ by identifying $v_i$ and $v_{i+2}$, where $i\in \{1,2\}$. 
	By the assumption and Corollary~\ref{cor:lem10}, 
	%if the number of triangles increases in both $G_1$ and $G_2$, 
	%then $G$ contains a triangle containing two vertices of $\alpha$, a contradiction. 
	we may assume, without loss of generality, that $G_1$ contains $T$ as the unique triangle.
	Note that the graph $H_1$ obtained from $H$ by identifying $v_1$ and $v_3$ contains $e$ and is thus $3$-colorable by the minimality. 
	Thus, we can extend the coloring of $H_1$ to the coloring of $H$ in which $v_1$ and $v_3$ receive the same color, a contradiction.

	\medskip \noindent
	{\bf Case 3:} \textit{$G$ has a $4$-face $\alpha=v_1v_2v_3v_4$ such that at most one vertex of $\alpha$ is incident with $T$ 
	and both vertices of $e$ are incident with $\alpha$.} \quad
	We may assume, without loss of generality, that $e = v_1v_3$. 
	Let $G_2$ be the graph obtained from $G$ by identifying $v_2$ and $v_{4}$. 
	Note that if the number of triangles does not increase in $G_2$, then we can continue as in Case 2.
	
	Therefore, by Lemma~\ref{lem:10Borodin}, there exist vertices $x,z\in V(G)$ such that $xv_4,xz,zv_2\in E(G)$. 
	Consequently, no $4$-face of $G$, other than $\alpha$, can contain both vertices $v_1$ and $v_3$ due to planarity.  
	%since $v_3$ would have to be incident with either $x$ or $z$, due to planarity. 
	%But, in such a case, there would be another triangle in $G$, $v_3xv_4$ or $v_3zv_2$, 
	%distinct from $T$ (since at most one vertex of $T$ is incident with $\alpha$).

	Due to Cases 1 and 2, and the fact that $\alpha$ contains both vertices of $e$, 
	there exists a $4$-face $\alpha'=v_1'v_2'v_3'v_4'$ such that $\alpha'$ contains two vertices of $T$, 
	say $v_1'$ and $v_2'$ (note that the two vertices incident with $T$ are not opposite in $\alpha'$, 
	otherwise there would be another triangle in $G$), 
	with $z'$ being the third vertex of $T$. 
	Let $G_i'$ be the graph obtained from $G$ by identifying $v_i'$ and $v_{i+2}'$, where $i\in \{1,2\}$. 
	Again, if the number of triangles does not increase in $G_1'$ or $G_2'$, then we can color $H$ with $3$-colors.
	
	It follows that there exist vertices $x',y'\in V(G)$ such that $x'z'$, $x'v_4'$, $y'z'$, and $y'v_3'\in E(G)$. 
	Suppose that at least one of $C_1 = z'v_2'v_3'y'$ or $C_2 = z'v_1'v_4'x'$ is a $4$-face, say $C_1$. 
	By our observation above, $C_1$ does not contain both vertices of $e$.
	Let $G'$ be the graph obtained from $G$ 
	by identifying $v_2'$ and $y'$.
	%by deleting $v_2'$ 
	%and adding an edge between $y'$ and every vertex in $N_G(v_2')\setminus N_G(y')$. 
	Note that the number of triangles in $G'$ does not increase. 	
	Let $H'$ be the graph obtained from $G'$ by adding the edge $e$. 
	By the minimality, we can color $H'$ with $3$ colors and extend the coloring to a coloring of $H$, 
	in which $y'$ and $v_2'$ receive the same color, a contradiction. 
	
	Thus, we may assume that both $C_1$ and $C_2$ are separating $4$-cycles.	
	Note that if the vertices of $\alpha$ (and thus also the endvertices of $e$) belong to the vertex set $V_1 = V(\exter(C_1)) \cup V(C_1)$, 
	then $H[V_1]$ contains both $T$ and $e$. 
	Therefore, we can color $H[V_1]$ by the minimality and extend the coloring of $C_1$ to a coloring of $H[V(\inter(C_1))\cup V(C_1)]$ by Theorem~\ref{thm:5face}. 
	We use an analogous argument for $C_2$ in the case when the vertices of $\alpha$ belong to the graph induced by the vertex set $V(\inter(C_1)) \cup V(C_1)$,
	which implies that the vertices of $\alpha$ belong to the vertex set $V(\exter(C_2)) \cup V(C_2)$.
	Thus, $H$ is $3$-colorable, a contradiction.

	\medskip \noindent
	{\bf Case 4:} \textit{$G$ has a $4$-face $\alpha=v_1v_2v_3v_4$ such that exactly two of its vertices, say $v_1$ and $v_2$, are incident with $T$,
	and at most one vertex of the edge $e$ is incident with $\alpha$.} \quad
	Let $z$ be the third vertex of $T$.
	Using similar arguments as in the previous cases, 
	we infer that there exist vertices $x,y\in V(G)$ such that $xz$, $xv_4$, $yz$, and $yv_3\in E(G)$.
	
	Suppose that $C_1=zv_2v_3y$ is a $4$-face. 
	If $e \neq v_2y$, then consider the graph $G'$ obtained from $G$ 
	by identifying $v_2$ and $y$.
	%by deleting $v_2$
	%and adding an edge between $y$ and every vertex in $N_G(v_2)\setminus N_G(y)$.	
	Note that the number of triangles in $G'$ does not increase.
	Let $H'$ be the graph obtained from $G'$ by adding the edge $e$. 	 
	By the minimality, we can color $H'$ with $3$-colors and extend the coloring to a coloring of $H$, 
	in which $y$ and $v_2$ receive the same color, a contradiction.	
	Therefore, $e = v_2y$. 
	But then, either $C_2=zv_1v_4x$ is a $4$-face, 
	in which case we can apply the same procedure on $v_1$ and $x$ as we did on $v_2$ and $y$,  
	or $C_2$ is a separating $4$-cycle. 
	However, since both $V(T)$ and $V(e)$ belong to the vertex set $V_1 = V(\exter(C_2)) \cup V(C_2)$, 
	we can complete the coloring in a similar manner as in the last paragraph of Case 3, a contradiction. 
	
	Thus, by symmetry, both $C_1$ and $C_2$ are separating $4$-cycles. 
	Moreover, each of $C_1$ and $C_2$ contains exactly one vertex of $e$ in its interior.
	Furthermore, $T$ is a $3$-face, otherwise we can color $H[V(\exter(T)) \cup V(T)]$ by the minimality, 
	and then extend the coloring to the interior of $T$ by Theorem~\ref{thm:5face}.
	Additionally, due to Case 1, 
	there exists a $4$-face $\alpha'=v_1'v_2'v_3'v_4'$ in $G$, distinct from $\alpha$.
	If identifying either $v_1'$ and $v_3'$, or $v_2'$ and $v_4'$ results in a graph with one triangle, namely $T$,
	then by the minimality, it is $3$-colorable and the coloring can be extended to $H$.
	Therefore, by the fact that $G$ has only one triangle and Lemma~\ref{lem:10Borodin},
	two vertices of $\alpha'$ are incident with $T$, say $v_1' = v_1$ and $v_2' = z$ 
	(meaning that at least one of $v_3'$ and $v_4'$ is in $V(\inter(C_2))$, see Figure~\ref{fig:case4}) 
	and there are vertices $x'$ and $y'$ in $G$ such that $x'v_4'$, $x'v_2$, $y'v_3'$, and $y'v_2 \in E(G)$.
	This is not possible due to the planarity of $G$, a contradiction.
	\begin{figure}[htb!]
		$$
			\includegraphics{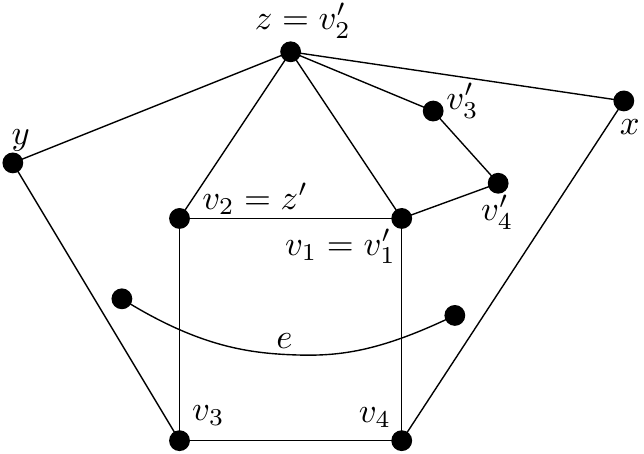}
		$$
		\caption{The $4$-faces $\alpha$ and $\alpha'$ in the last part of Case 4.}
		\label{fig:case4}
	\end{figure}

	\medskip \noindent
	{\bf Case 5:} \textit{$G$ has at least three $4$-faces and each of them is incident with two vertices of $T$ and both vertices of $e$.} \quad	
	Let $\alpha=v_1v_2v_3v_4$ be such a face and let $T=v_1v_2z$. 
	Without loss of generality, we may assume that $e = v_1v_3$. 
	Let $G_2$ be the graph obtained from $G$ by identifying $v_2$ and $v_{4}$. 
	If the number of triangles does not increase in $G_2$, then we are done. 
	Thus, by Lemma~\ref{lem:10Borodin}, there exists a vertex $x \in V(G)$ such that $xz$, $xv_4\in E(G)$. 
	Note that by the assumptions, $C=zv_1v_4x$ is not a $4$-face, since it is incident to exactly one vertex of $e$. 
	Therefore, $C$ is a separating $4$-cycle. 
	But then, the vertices of both $T$ and $e$ are contained in the vertex set $V_1 = V(\exter(C)) \cup V(C)$. 
	Let $V_2 = V(\inter(C)) \cup V(C)$. 
	By the minimality, we can color $G[V_1]$ and extend the coloring of $C$ to the coloring of $G[V_2]$ by Theorem~\ref{thm:5face}, a contradiction. 
	
	\medskip
	Since no $4$-face can be incident with all three vertices of $T$, the proof is completed.
\end{proof}

In the second step of proving Theorem~\ref{thm:triangle_plus_precolored_pair},
we show that any two non-adjacent vertices in a planar graph with one triangle can be colored with the same color.
\begin{theorem}
	\label{thm:triangle_plus_same_color_pair}
	Let $G$ be a planar graph with at most one triangle.
	Then each coloring of any two non-adjacent vertices with the same color can be extended to a $3$-coloring of $G$.
\end{theorem}

\begin{proof}
	We prove the theorem by contradiction.
	Suppose that a counterexample $G$ is a plane graph with the minimum number of vertices.
	By Theorem~\ref{thm:precolor2verts}, we may also assume that $G$ contains exactly one triangle $T$.
	Let $u$ and $v$ be two non-adjacent vertices of $G$.% and let $\phi$ be a coloring of $u$ and $v$ such that $\phi(u) = \phi(v)$. 

	Let $H$ be the graph obtained from $G$ by identifying the vertices $u$ and $v$. 	 
	Clearly, $n_G = n_H + 1$ and $m_G = m_H$. 
	By the minimality, $H$ is $4$-critical.
	To reach a contradiction, we only need to prove that $H$ is $3$-colorable, 
	which implies that there exists a $3$-coloring of $G$ in which $u$ and $v$ receive the same color.
	
	We consider three cases regarding $4$-faces in $G$.

	\medskip \noindent
	{\bf Case 1:} \textit{$G$ has no $4$-faces.} \quad 
	By the Handshaking Lemma, we have
	$$
		2m_G = \sum_{\alpha \in F(G)} \ell(\alpha) \ge 3 + 5 \cdot (f_G - 1) = 5f_G - 2\,.
	$$	
	Then, $5f_G \le 2m_G + 2$ and by applying the Euler's Formula, we infer that 
	$$
		10 = 5n_G - 5m_G + 5f_G \le 5n_G - 3m_G + 2 = 5n_H + 5 - 3m_H + 2\,.
	$$
	Thus, 
	$$
		m_H \le \frac{5n_H - 3}{3}\,,
	$$ 
	a contradiction to Theorem~\ref{thm:4-critical}.	

	\medskip \noindent
	{\bf Case 2:} \textit{$G$ has exactly one $4$-face.} \quad
	Similarly as in Case 1, we can compute that $5f_G \le 2m_G + 3$ 
	and by applying Euler's Formula, we infer that
	$$
		m_H \le \frac{5n_H - 2}{3}\,.
	$$ 
	In the case when $m_H < \frac{5n_H - 2}{3}$, we obtain a contradiction to Theorem~\ref{thm:4-critical},  
	and therefore, $H$ has exactly $\frac{5n_H - 2}{3}$ edges. 
	
	Let $\alpha=v_1v_2v_3v_4$ be the $4$-face in $G$ and 
	let $G_i$ be the graph obtained from $G$ by identifying $v_i$ and $v_{i+2}$, where $i \in \{1,2\}$. 

	Suppose first that the number of triangles does not increase in $G_1$ or $G_2$, say $G_1$.
	In the case $\{u,v\} \neq \{v_1,v_3\}$, we identify $v_1$ and $v_3$ in $H$ to obtain the graph $H_1$.
	By the minimality, we can color $H_1$ with $3$ colors and extend the coloring to a coloring of $H$, 
	and therefore also to $G$, a contradiction.
	Hence, we may assume that $\{u,v\} = \{v_1,v_3\}$. 
	In this case, $H$ is a planar graph with exactly one triangle. 
	Thus, by Theorem~\ref{thm:3triangles}, there exists a $3$-coloring of $H$, and therefore also of $G$, a contradiction.

	We may thus assume that the number of triangles increases in both $G_1$ and $G_2$. 
	By Lemma~\ref{lem:10Borodin}, without loss of generality, we may assume that there exist vertices $x,y,z \in V(G)$ such that 
	$zv_1$, $zv_2$, $xz$, $xv_4$, $yz$, and $yv_3 \in E(G)$, where $zv_1v_2$ is $T$. 	
	Since $G$ contains exactly one $4$-face, it follows that both $C_1 = zv_1v_4x$ and $C_2 = zv_2v_3y$ are separating $4$-cycles. 
	
	Note that if both $u$ and $v$ belong to the subgraph of $G$ induced by the vertex set $V_1 = V(\exter(C_1)) \cup V(C_1)$, 
	then we can color $G[V_1]$ by the minimality and use Theorem~\ref{thm:5face} to extend the coloring of $C_1$ to the coloring of the interior of $C_1$. 	
	By symmetry, we may thus assume, without loss of generality, that $u \in V(\inter(C_1))$ and $v \in V(\inter(C_2))$. 
	
	Since $m_H = \frac{5n-2}{3}$, by Theorems~\ref{thm:4-critical_ore} and~\ref{thm:pl44f-graphs},
	we infer that $H$ must have at least $5$ triangles. 	
	Therefore, since $G$ has exactly one triangle, it follows that by identifying $u$ and $v$, we create at least four new triangles. 
	We will prove that this cannot happen. 
	
	First, observe that no new triangle can contain vertices $x$ or $y$, 
	since that would imply the existence of another triangle, distinct from $T$, in $G$. 
	Next, observe that $u$ is adjacent with at most one of the vertices $v_1$ and $v_4$, 
	and $v$ is adjacent with at most one of the vertices $v_2$ and $v_3$. 
	Thus, at most one new triangle can be formed using the edges $v_1v_2$ or $v_3v_4$, 
	and so there must exist at least three triangles in $H$ which contain the vertex $z$ and either $u$ or $v$, say $u$, is adjacent to $z$.
	Therefore, there exist at least three vertices $w_1,w_2,w_3 \in V(G)$ such that $w_1,w_2,w_3 \in V(\inter(C_2))$.
	Moreover, each of them is adjacent to $z$ and $v$ (see Figure~\ref{fig:caseB2}). 
	\begin{figure}[htb!]
		$$
			\includegraphics{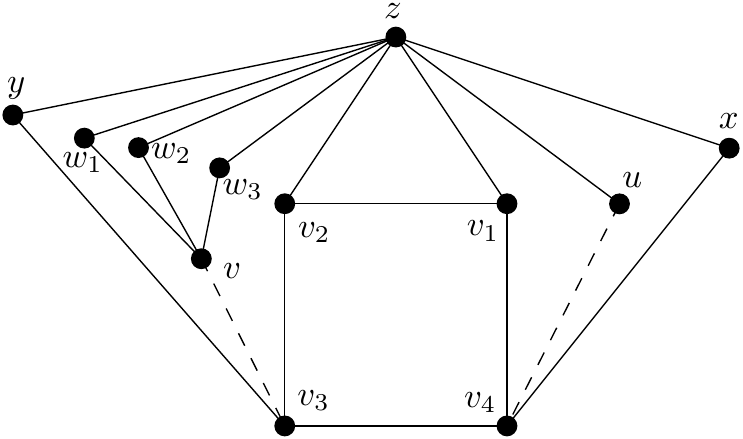}
		$$
		\caption{The vertices in $G$ comprising triangles in $H$ in the last part of Case 2.}
		\label{fig:caseB2}
	\end{figure}	
	Consider now the $4$-cycle $C=zw_1vw_2$. 
	Since $G$ contains exactly one $4$-face, it follows that $C$ is a separating $4$-cycle. 
	Furthermore, the exterior of $C$ together with the vertices of $C$ contains both $u$ and $v$, as well as $T$. 
	Thus, by the minimality, we can color $G[V(\exter(C)) \cup V(C)]$
	and extend the $3$-coloring of the vertices of $C$ to a $3$-coloring of $H$ by Theorem~\ref{thm:5face}.

	\medskip \noindent
	{\bf Case 3:} \textit{$G$ has at least two $4$-faces.} \quad 
	Let $\alpha=v_1v_2v_3v_4$ be a $4$-face 
	and let $G_i$ be the graph obtained from $G$ by identifying $v_i$ and $v_{i+2}$, where $i \in \{1,2\}$. 
	Since $G$ contains exactly one triangle, by Corollary~\ref{cor:lem10}, 
	either, without loss of generality, $v_1v_2$ is an edge of $T$ 
	or we can identify $v_1$ and $v_3$ or $v_2$ and $v_4$ without creating any new triangles. 
	Suppose first that $v_1v_2$ is not an edge of $T$; say that $G_1$ has at most one triangle. 
	Then, in the case $\{u,v\} \neq \{v_1,v_3\}$, we identify $v_1$ and $v_3$ in $H$ to obtain the graph $H_1$.
	By the minimality, we can color $H_1$ with $3$ colors and extend the coloring to a coloring of $H$, 
	and therefore also to $G$, a contradiction.
	Hence, we may assume that $\{u,v\} = \{v_1,v_3\}$. 
	In this case, $H$ is a planar graph with exactly one triangle. 
	Thus, by Theorem~\ref{thm:3triangles}, there exists a $3$-coloring of $H$, and therefore also of $G$, a contradiction.
	
	Thus, we may assume that $T = v_1v_2z$, with $z$ being distinct from $v_3$ and $v_4$, 
	and that in both $G_1$ and $G_2$ the number of triangles is at least $2$. 
	Therefore, by Lemma~\ref{lem:10Borodin}, there exist vertices $x,y\in V(G)$ such that $xz$, $xv_4$, $yz$, and $yv_3 \in E(G)$.

	Suppose that $C_1=zv_1v_4x$ is a $4$-face. 
	Then, due to planarity of $G$, 
	in the graph $G'$ obtained by identifying $v_1$ and $z$, no new triangle is created.
	%it is not difficult to see that 	
	%the property $(b)$ of Corollary~\ref{cor:lem10} holds for $C_1$. 
	%Thus, we can use minimality of the graph obtained from $G$ by identifying the two opposite vertices of $C_1$, 
	%that do not create any new triangles,
	Thus, by the minimality, we can color $G'$ and infer $3$-colorability of $G$ in a similar manner as above, a contradiction.

	Therefore, by symmetry, we may assume that both $C_1$ and $C_2=zv_2v_3y$ are separating $4$-cycles.
	Note that if both $u$ and $v$ belong to the vertex set $V_1 = V(\exter(C_1))\cup V(C_1)$ 
	(resp., $V_2 = V(\exter(C_2)) \cup V(C_2)$), 
	then, by the minimality, we can color the graph $H_1$ (resp., $H_2$) 
	obtained from $G[V_1]$ (resp., $G[V_2]$) by identifying $u$ and $v$ 
	and extend the coloring to a coloring of $H$ by Theorem~\ref{thm:5face}, hence also obtaining a $3$-colorability of $G$. 
	
	Thus, we may assume, without loss of generality, that $u \in \inter(C_1)$ and $v \in \inter(C_2)$. 
	Now, consider a $4$-face $\alpha'=v_1'v_2'v_3'v_4'$. 
	If $\alpha'$ satisfies the property $(b)$ of Corollary~\ref{cor:lem10}, then we proceed as above to obtain a contradiction.
	Therefore, $\alpha'$ is incident with $T$ and, by planarity of $G$,
	the vertices of $\alpha'$ are all contained in $V(\inter(T)) \cup V(T)$.
	But then, both $u$ and $v$ belong to the exterior of $T$ and 
	we can color, by the minimality, the graph obtained from $G[V(\exter(T) \cup V(T))]$ by identifying $u$ and $v$.
	Finally, we extend the obtained coloring to a coloring of $H$ by Theorem~\ref{thm:5face}. 
	Hence, from the coloring of $H$, we again obtain $3$-colorability of $G$, a contradiction.	
	This completes the proof.
\end{proof}

Theorems~\ref{thm:triangle_plus_edge} and~\ref{thm:triangle_plus_same_color_pair} combined settle Theorem~\ref{thm:triangle_plus_precolored_pair}.
Next, we prove Theorem~\ref{thm:onetriangle_5face}.
\begin{proof}[Proof of Theorem~\ref{thm:onetriangle_5face}]
	Let $G$ be a planar graph with at most one triangle and let $f$ be a precolored face of length at most $4$.
		
	Suppose first that $f$ is of length $3$.
	Since there is only one coloring of $f$ (up to a permutation of colors), the result follows from Theorem~\ref{thm:3triangles}.
	
	Thus, we may assume that $f = v_1v_2v_3v_4$ is a $4$-face.	
	Suppose that the precoloring of $f$ uses all three colors.
	Then, two non-adjacent vertices of $f$, say $v_1$ and $v_3$, receive distinct colors
	and the other two vertices are colored with the third.
	Note that the same coloring of $f$ (up to a permutation of colors) can be obtained by adding an edge between $v_1$ and $v_3$.
	The obtained graph is $3$-colorable by Theorem~\ref{thm:triangle_plus_edge}.
	
	Therefore, we may assume that the vertices of $f$ are precolored with two colors.
	We proceed by contradiction. 
	Let $G$ be a plane graph with at most one triangle such that a precoloring of some $4$-face $f$ with two colors cannot be extended to a $3$-coloring of $G$.
	Moreover, let $G$ be the smallest such graph in terms of the vertices.
	Clearly, $G$ has exactly one triangle $T$, otherwise the precoloring can be extended by Theorem~\ref{thm:5face}.
	
	Let $G_i$ be the graph obtained from $G$ by identifying $v_i$ and $v_{i+2}$, where $i \in \{1,2\}$. 	
	If the number of triangles does not increase in $G_1$ or $G_2$, say $G_1$,
	then there is a $3$-coloring of $G_1$, guaranteed by Theorem~\ref{thm:triangle_plus_precolored_pair}, 
	which induces a $3$-coloring of $G$ such that the vertices of $f$ are colored with two colors.
	
	Thus, by Lemma~\ref{lem:10Borodin}, without loss of generality, 
	we may assume that there exist vertices $x,y,z \in V(G)$ such that 
	$zv_1$, $zv_2$, $xz$, $xv_4$, $yz$, and $yv_3 \in E(G)$, where $T=zv_1v_2$.
	Observe that coloring of $f$ forces also the colors on $x$, $y$, and $z$ (see Figure~\ref{fig:4ext}).
	\begin{figure}[htb!]
		$$
			\includegraphics{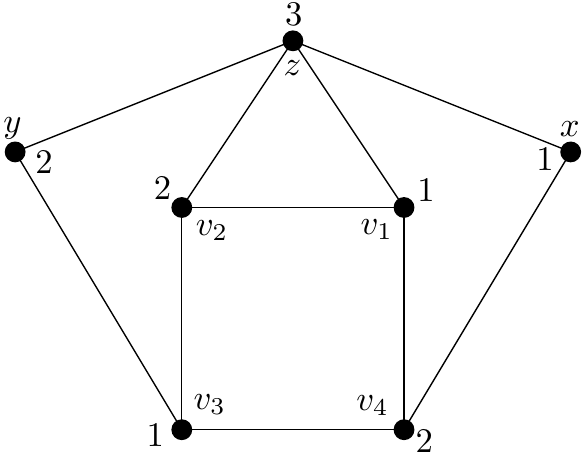}
		$$
		\caption{The coloring of $f$ forces the colors of $x$, $y$, and $z$.}
		\label{fig:4ext}
	\end{figure}
	
	Let $C_1=zv_1v_4x$ and $C_2=zv_2v_3y$.
	Suppose that at least one of $C_1$ or $C_2$, say $C_1$, is a separating $4$-cycle.
	Then, by the minimality, the coloring of $f$ extends to a $3$-coloring of $G[V(\exter(C_1)) \cup V(C_1)]$.
	Since the obtained coloring of $C_1$ extends to a $3$-coloring of $G[V(\inter(C_1)) \cup V(C_1)]$ by Theorem~\ref{thm:5face},
	we obtain a $3$-coloring of $G$, a contradiction.
	
	Thus, we may assume that both $C_1$ and $C_2$ are $4$-faces in $G$. 
	In a similar manner as above, we infer that $T$ must be a $3$-face.
	But then, the precoloring of the $5$-cycle $C_3=v_3v_4xzy$ given in Figure~\ref{fig:4ext} 
	extends to a $3$-coloring of $G[V(\exter(C_3) \cup V(C_3))]$ (which might as well be an empty graph) 
	by Theorem~\ref{thm:5face}
	and we color the two vertices in the interior of $C_3$ as in Figure~\ref{fig:4ext},
	hence obtaining a $3$-coloring of $G$, a contradiction.
	This completes the proof.
\end{proof}

We conclude this section with a proof of Theorem~\ref{thm:3neighbors}.
\begin{proof}[Proof of Theorem~\ref{thm:3neighbors}]
	We prove the theorem by contradiction.
	Let $G$ be a minimal counterexample to the theorem, i.e.,
	$G$ is a $K_4'$-free planar graph with at most one triangle and the minimum number of vertices
	such that there is a vertex $u$ of degree at most $3$ with an independent neighborhood,
	such that precoloring the vertices in $N(u)$ with a same color does not extend to a $3$-coloring of $G$.

	First, observe that by Theorem~\ref{thm:4vert}, $G$ has exactly one triangle $T$,
	and by Theorem~\ref{thm:triangle_plus_precolored_pair}, $u$ is a $3$-vertex.
	Let $N[u] = \{u,u_1,u_2,u_3\}$ 
	and let $H$ be the graph obtained by identifying $N[u]$ into a vertex $w$.
	Let $\alpha_1$, $\alpha_2$, and $\alpha_3$ be the three faces incident to $u$ in $G$ 
	that contain respectively $\{u_1,u_2\}$, $\{u_2,u_3\}$, and $\{u_1,u_3\}$. 
	Furthermore, let $\alpha_1'$, $\alpha_2'$, and $\alpha_3'$ be the faces incident to $w$ in $H$ 
	corresponding to $\alpha_1$, $\alpha_2$, and $\alpha_3$.
	
	Clearly, every $3$-coloring of $H$ induces a $3$-coloring of $G$ with $u_1$, $u_2$, and $u_3$ colored with a same color, 
	while $u$ can be colored with either of the remaining two colors.
	Additionally, since $G$ is a planar graph, $H$ is also a planar graph 
	and by the minimality of $G$, $H$ is $4$-critical.
	Observe also that $n_G = n_H + 3$, $m_G = m_H + 3$, and $f_G = f_H$.

	Now, we prove two structural properties of $H$.
	\begin{claim}
		\label{claim:no_separating_C3}
		$H$ has no separating triangles.
	\end{claim}

	\begin{proofclaim}
		Suppose the contrary and let $C$ be a separating triangle in $H$. 
		First, suppose that $C$ is the triangle of $G$.
		Without loss of generality, we may assume that $w \in V(\inter(C))$.
		By the minimality, there is a $3$-coloring of $H[V(\inter(C)) \cup V(C)]$,
		and by Theorem~\ref{thm:3triangles}, we can extend it to a $3$-coloring of $H$, 
		since $H[V(\exter(C)) \cup V(C)]$ has exactly one triangle, a contradiction.

		Therefore, we may assume that $C \neq T$.	
		In that case, $C$ has been created from a $5$-cycle $C_G$ after we identified $N[u]$ into $w$ and thus $w \in V(C)$. 		 
		Since $C \neq T$, we may assume, without loss of generality,
		that $H[V(\inter(C)) \cup V(C)]$ contains $\alpha_1'$ but not $\alpha_2'$ or $\alpha_3'$ (see Figure~\ref{fig:thm11_sep3}). 	 
		\begin{figure}[htb!]
			$$
				\includegraphics{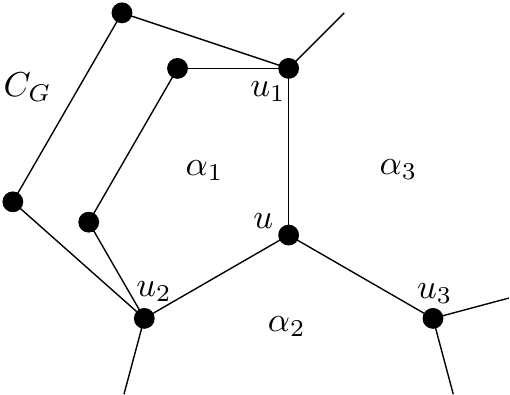}
			$$
			\caption{A separating $5$-cycle in $G$ containing $\alpha_1$.}
			\label{fig:thm11_sep3}
		\end{figure}	

		By the minimality, there is a $3$-coloring $\phi$ of $H[V(\exter(C)) \cup V(C)]$.
		Now, we show that we can extend $\phi$ to the interior of $C$. 
		Let $H_1 = H[V(\inter(C)) \cup V(C)]$.
		We proceed by induction on the number of separating triangles in $H_1$.
		First, recall that all separating triangles in $H_1$ are incident to $w$; 
		more precisely, they were obtained from $5$-cycles in $G$ containing $\{u_1,x,u_2\}$.

		Suppose that $H_1 = H[V(\inter(C)) \cup V(C)]$ has no separating triangle. 
		Then it has at most three triangles: $C$ as its outer face, possibly $\alpha_1'$, and possibly $T$. 
		Therefore, $H_1$ is a planar graph with at most three triangles and thus $3$-colorable by \Cref{thm:3triangles}.

		So, we may assume that $H_1$ has at least one separating triangle; 
		we select a separating triangle $C'$ such that all separating triangles in $H_1$ are contained in $H_1' = H_1[V(\inter(C')) \cup V(C')]$.
		Then, by induction, there is a $3$-coloring $\phi'$ of $H_1'$.
		Finally, using the colorings $\phi$ and $\phi'$, 
		we can complete the coloring of $H$ by coloring $H[V(H_1) \setminus V(\inter(C'))]$
		using \Cref{thm:3triangles} and an eventual permutation of colors in $\phi'$, 
		a contradiction. 
	\end{proofclaim}

	\begin{claim}
		\label{claim:no_separating_C4}
		If $H$ has a separating $4$-cycle, then both its interior and exterior must contain $w$ or a triangle.
	\end{claim}

	\begin{proofclaim}
		Suppose the contrary and let $C$ be a separating $4$-cycle of $H$ 
		such that $H[V(\inter(C)) \cup V(C)]$ is a triangle-free planar graph that does not contain $w$. 
		By the minimality, there is a $3$-coloring $\phi$ of $H[V(\exter(C)) \cup V(C)]$. 
		By \Cref{thm:5face}, we can extend $\phi$ to the whole graph $H$, a contradiction.
	\end{proofclaim}

	Now, we are ready to finish the proof 
	by considering three cases regarding $4$-faces of $G$.

	\medskip \noindent
	{\bf Case 1:} \textit{$G$ has no $4$-faces.} \quad
	By the Handshaking Lemma, we have $2m_G \ge 5f_G - 2$ and so $2m_H + 6 \ge 5f_H - 2$.
	Then, $5f_H \le 2m_H + 8$ and by applying the Euler's Formula on $G$, we infer that 
	$$
		m_H \le \frac{5n_H - 2}{3}\,.
	$$ 
	Since $H$ is $4$-critical, by \Cref{thm:4-critical_ore}, we have that $m_H= \frac{5n_H - 2}{3}$ and that $H$ is a $4$-Ore graph.
	Moreover, since $H$ does not have separating triangles by \Cref{claim:no_separating_C3}, 
	there are at most four triangles in $H$ ($T$ and the faces $\alpha_1'$, $\alpha_2'$, and $\alpha_3'$). 	
	Thus, by \Cref{thm:3triangles}, $H$ has exactly four triangles 
	and  by \Cref{thm:pl44f-graphs}, $H$ is a $\Pl$-graph. 	
	Recall that three of the triangles are incident to the same vertex $w$. 
	The only $\Pl$-graph for which this is true is $K_4$~\cite[Theorem~4]{BorDvoKosLidYan14}. 
	However, to obtain $K_4$, all three neighbors of $u$ in $G$ must be of degree $2$,
	meaning that $G$ must be $K_4'$, a contradiction.
	%contains a subgraph isomorphic to $K_4'$, a contradiction.

	\medskip \noindent
	{\bf Case 2:} \textit{$G$ has a $4$-face that is incident to $u$.} \quad
	As a result, after identifying $u$ and its neighbors, 
	$H$ has at most three triangles by \Cref{claim:no_separating_C3}. 
	Therefore, $H$ is $3$-colorable by \Cref{thm:3triangles}.

	\medskip \noindent
	{\bf Case 3:} \textit{$G$ has a $4$-face $\alpha=v_1v_2v_3v_4$ and $\alpha$ is not incident to $u$.} \quad 
	The edges $v_1v_3$ and $v_2v_4$ are not present in $G$, otherwise $G$ would have at least two triangles.
	Moreover, if $u$ is adjacent to two (opposite) vertices of $\alpha$, say $v_1=u_1$ and $v_3=u_3$, 
	then, by Case 2, neither $C_1=uv_1v_2v_3$ nor $C_2=uv_1v_4v_3$ is a $4$-face.
	Moreover, without loss of generality, we may assume that $u_2 \in V(\exter(C_1))$.
	However, by the minimality, there is a $3$-coloring of $G[V(\exter(C_1)) \cup V(C_1)]$,
	and it can easily be extended to the whole $G$ by Theorem~\ref{thm:onetriangle_5face}.
	Therefore, at most one of the vertices of $\alpha$ is adjacent to $u$.
	
	Let $G_i$ be the graph obtained from $G$ by identifying $v_i$ and $v_{i+2}$, where $i \in \{1,2\}$.	
	Suppose that the only triangle in $G_1$ is $T$. 
	Then, by the minimality, the graph $H_1$ obtained by identifying the vertices $v_1$ and $v_3$ in $H$ is $3$-colorable. 
	Clearly, its coloring can be extended to $H$ and thus also to $G$, a contradiction.
	
	Therefore, by symmetry, we may assume that in $G_1$ and $G_2$ the number of triangles increases.
	It follows by Lemma~\ref{lem:10Borodin} that there are vertices $x$, $y$, and $z$ 
	such that $v_1z,v_2z,xz,xv_4,yz,yv_3 \in E(G)$, where $T=v_1v_2z$.
	If one of $C_1=zv_1v_4x$ and $C_2=zv_2v_3y$ is a $4$-face, 
	it has the same properties as $\alpha$ and two of its vertices are incident with $T$. 
	But that is not possible due to planarity. 
	
	Thus, $C_1$ and $C_2$ are separating $4$-cycles of $G$. 
	Since, at most one of them can contain $u \neq z$ (by definition, $u$ is not incident with a triangle), 
	the other one remains a separating $4$-cycle of $H$, 
	which does not contain $T$ nor $w$, a contradiction to \Cref{claim:no_separating_C4}.
	This completes the proof.
\end{proof}

\section{Conclusion}
	\label{sec:conclude}

One motivation for the research presented in this paper was a conjecture 
on adynamic coloring of planar graphs with one triangle.
An {\em adynamic coloring} is a proper vertex coloring of a graph $G$ 
such that for at least one $2^+$-vertex all of its neighbors are colored with a same color.
Clearly, to admit such a coloring, $G$ must have at least one $2^+$vertex $v$ with an independent neighborhood,
i.e., $v$ is not incident to a triangle. 
This is also a sufficient condition.

In~\cite{SurLuzMad20}, it was proved that every triangle-free planar graph admits an adynamic $3$-coloring 
(note that this fact is also a corollary of Theorem~\ref{thm:4vert}).
On the other hand, there are planar graphs with two triangles that need $4$ colors (see, e.g., the graph in Figure~\ref{fig:2trian}).
Regarding planar graphs with one triangle, \v{S}urimov\'{a} et al.~\cite{SurLuzMad20} conjectured 
that they are $3$-colorable as soon as they contain a $2^+$-vertex with an independent neighborhood.
Using the results of this paper, we are able to answer the conjecture in affirmative.
\begin{theorem}
	\label{thm:adynamic_one_triangle}
	Every planar graph with at most one triangle and a $2^+$-vertex with an independent neighborhood 
	is adynamically $3$-colorable.
\end{theorem}

\begin{proof}
	We again proceed by contradiction.
	Let $G$ be a minimum counterexample in terms of the number of vertices with some fixed embedding.
	By Theorem~\ref{thm:4vert}, $G$ has exactly one triangle $T$.
	Suppose first that there is a $2$-vertex $v$ in $G$ and let $N(v) = \{v_1,v_2\}$.
	The graph $G'$ obtained by splitting $v$ into two adjacent vertices both connected to $v_1$ and $v_2$ 
	is planar with at most three triangles and thus $3$-colorable by Theorem~\ref{thm:3triangles}.
	Its coloring induces a coloring of $G$ in which $v_1$ and $v_2$ receive the same color, a contradiction.
	
	Therefore, $\delta(G) \ge 3$.
	Moreover, by the Handshaking Lemma and the Euler's Formula, there are at least nine $3$-vertices in $G$,
	and so at least six $3$-vertices are not incident with $T$.	
	Hence, by Theorem~\ref{thm:3neighbors}, $G$ contains a subgraph $D$ isomorphic to $K_4'$.
	Since $\delta(G) \ge 3$, there is a $5$-face $\alpha$ of $D$ that is not a face in $G$.
	
	The graph induced by the interior of $\alpha$ in $G$ and $V(\alpha)$ is a triangle-free plane graph,
	which we can $3$-color adynamically.
	This coloring gives us a coloring of the vertices of $\alpha$ and fixes also the color of the vertex of $T$
	that is not incident with $\alpha$.
	It remains to color (eventual) interiors of the other two $5$-faces of $D$ in $G$ 
	and the interior of $T$. All can be colored by Theorem~\ref{thm:5face}.
	This completes the proof.
\end{proof}

\medskip
Regarding other results in this paper, there are a number of possibilities for further work.
For example, similar results for planar graphs with two triangles and three triangles would be interesing.
\begin{problem}
	Characterize planar graphs with two (resp., three) triangles, 
	in which precoloring of any two non-adjacent vertices extends to a $3$-coloring of the graph.
\end{problem}

Also, one could investigate in more details precoloring extensions from larger independent sets.
\begin{problem}
	\label{prob:three}
	Characterize planar graphs with one triangle, 
	in which precoloring of any three non-adjacent vertices extends to a $3$-coloring of the graph.	
\end{problem}
Additionally, Problem~\ref{prob:three} could be extended to determining the properties of triples of non-adjacent vertices
whose precoloring does not extend to the whole graph; particularly, which colorings of them. 

We showed that a precoloring of a $5$-face in a planar graph $G$ with one triangle cannot always be
extended to a $3$-coloring of $G$. 
So it is natural to ask for a characterization similar to characterizations for faces of lengths $6$ to $9$ 
in triangle-free planar graphs.
\begin{problem}
	Characterize planar graphs with one triangle, 
	in which precoloring of a $5$-face (resp., $k$-face for any $k \ge 6$) extends to a $3$-coloring of the graph.
\end{problem}

%Finally, in Theorem~\ref{thm:triangle_plus_precolored_pair}, 
%we showed that precoloring of any two non-adjacent vertices in a planar graph $G$ with at most one triangle
%can be extended to a $3$-coloring of $G$.

\paragraph{Acknowledgment.}
The first author was partially supported by the grant HOSIGRA funded by the French National Research Agency
(ANR, Agence Nationale de la Recherche) under the contract number ANR--17--CE40--0022.
The second and the third author were partially supported by the Slovenian Research Agency program P1--0383 and the project J1--1692.
The third author also acknowledges support by the Young Researchers Grant of the Slovenian Research Agency.

\bibliographystyle{plain}
{
	\bibliography{References}
}

\end{document}